\documentclass[12pt]{amsart}
\usepackage{amsmath,amstext,amsfonts,amsthm,amssymb,epigraph}
\usepackage{eucal}
\usepackage{framed}
\usepackage{xcolor}
\usepackage{bbm}
\usepackage[all]{xy}                                        %
\swapnumbers
\newtheorem{thm}[subsubsection]{Theorem}
\newtheorem{lem}[subsubsection]{Lemma}
\newtheorem{prp}[subsubsection]{Proposition}
\newtheorem{crl}[subsubsection]{Corollary}
\newtheorem{THM}[subsection]{Theorem}

{  \theoremstyle{definition}
           \newtheorem{dfn}[subsubsection]{Definition}
           \newtheorem{exm}[subsubsection]{Example}

}

\newcommand{\an}{\mathit{an}}

\newcommand{\End}{\operatorname{End}}

\newcommand{\Hom}{\mathrm{Hom}} 

\newcommand{\id}{\mathrm{id}}

\renewcommand{\int}{\mathit{int}}

\newcommand{\iso}{\mathit{iso}}
\newcommand{\Ker}{\mathrm{Ker}}

\newcommand{\mx}{\mathit{max}}

\newcommand{\Map}{\operatorname{Map}}
\mathchardef\mhyphen="2D

\newcommand{\rad}{\operatorname{rad}}

\newcommand{\Span}{\operatorname{Span}}

\newcommand{\spec}{\mathtt{spec}}
\newcommand{\Stab}{\mathtt{Stab}}

\newcommand{\bA}{\mathbb{A}}

\newcommand{\cA}{\mathcal{A}}

\newcommand{\cO}{\mathcal{O}}

\newcommand{\cR}{\mathcal{R}}

\newcommand{\fg}{\mathfrak{g}}
\newcommand{\fh}{\mathfrak{h}}

\newcommand{\fm}{\mathfrak{m}}

\newcommand{\C}{\mathbb{C}}

\newcommand{\Z}{\mathbb{Z}}

\begin{document}

\title[]{Around the center}
\author{M.~Gorelik}
\address{Weizmann Institute of Science}
\email{maria.gorelik@gmail.com}
\author{V.~Hinich}
\address{University of Haifa}
\email{vhinich@gmail.com}
\author{V. Serganova}
\address{UC Berkeley}
\email{serganov@math.berkeley.edu}
 
\begin{abstract} 
The center of a semisimple Lie algebra can be described as the algebra of $W$-invariant functions on the dual of the Cartan subalgebra. The centers of many 
Lie superalgebras have a similar description, but the defining equivalence relation on the dual of the Cartan subalgebra
is not given by a finite group action.
Lagrangian equivalence relations that we introduce generalize the action of a subgroup of the orthogonal group. Using them, we present a new proof
of a result by Ian Musson about the centers of Lie superalgebras. Our proof is not based on a case-by-case analysis.
\end{abstract}
\maketitle

\section{Introduction}

\subsection{}

Let $V$ be a complex vector space and $R\subset V\times V$ be an equivalence relation. We denote by $\C[V]=SV^*$ the ring of polynomial
functions on $V$ and by $\C[V]^R$ the subring of invariant functions,
$\C[V]^R=\{f\in\C[V]|\ f(x)=f(y)\ \forall (x,y)\in R\}$.

This notion generalizes the invariant subring $\C[V]^G$ where $G$
is a group of automorphisms of $V$: the group action defines an equivalence relation and the two notions of invariant subring coincide.
In this case the quotient $V/G$ is an affine variety
with the ring of functions $\C[V]^G$. In particular, the set of 
maximal ideals $\spec\ \C[V]^G$ identifies with the set of $G$-orbits
in $V$. 
 
Let $\fg$ be a semisimple Lie algebra. The center $Z\fg$
of the enveloping algebra of $\fg$ identifies, via Harish-Chandra map, with the invariant subring $\C[V]^W$
where $V=\fh^*$ is dual to the Cartan subalgebra and $W$ the Weyl group. Therefore the maximal ideals of $\C[V]^W$ describe the central characters 
and they correspond to $W$-orbits in $\fh^*$.

Having in mind an application to the description of the centers of 
enveloping algebras of finite dimensional Lie superalgebras, we study
more general quotients $V/R$.

\subsection{Generalities}
There is a correspondence between relations $R\subset V\times V$ and
subrings of $\C[V]$, similar to the relation between subgroups
of the Galois group and field extensions in the Galois theory.

To any relation $R\subset V\times V$ we assign the subring
$\cA(R):=\C[V]^R$; in the opposite direction, we assign to any subring
$A\subset\C[V]$ the relation $\cR(A)$ consisting of all pairs $(x,y)$
such that $f(x)=f(y)$ for all $f\in A$.

It is obvious that $\cA(\cR(A))\supset A$ for any $A\subset\C[V]$ and
$\cR(\cA(R))\supset R$ for any $R\subset V\times V$. Moreover, it is
standard that $\cA(\cR(\cA(R))=\cA(R)$ and $\cR(\cA(\cR(A))=\cR(A)$. 

\begin{dfn}
\label{dfn:sat-detect}
\begin{itemize}
\item[1.] A subring $A\subset\C[V]$ is called {\sl saturated}
if $A=\cA(\cR(A))$.
\item[2.]A relation $R\subset V\times V$ is called {\sl saturated} if
$R=\cR(\cA(\cR(R))$.

\end{itemize}
\end{dfn}
 Saturated relations are automatically equivalence relations.
There is a one-to-one correspondence between
saturated equivalence relations and 
saturated subrings. 
 
\begin{dfn}
An equivalence relation $R$  on $V$ is called  {\sl detectable} if the embedding 
$\C[V]^R\to\C[V]$ establishes a bijection
between the set of maximal ideals of $\C[V]^R$ and the quotient $V/R$.
In other words, for any maximal ideal $\fm\subset\C[V]^R$ the set
of common zeroes of $f\in\fm$ is an equivalence class with respect to $R$ (in particular, this set is nonempty).
\end{dfn}

Note the following

\begin{lem}
\label{lem:detsat}
Any detectable relation is saturated.
A saturated relation $R$ is detectable  if, in addition, 
for every maximal ideal $\fm\subset\C[V]^R$
its extension $\C[V]\fm$ is a proper ideal in $\C[V]$.
\end{lem}
\begin{proof}
Let $(x,y)\not\in R$ and let $\fm_x$, $\fm_y$ be the maximal ideals
of $\C[V]^R$ defined by the equivalence classes of $x$ and $y$ respectively. Since $\fm_x\ne\fm_y$, there is $f\in\fm_x\setminus\fm_y$.
Then $f(x)=0$ but $f(y)\ne 0$. The converse is obvious as, if $\C[V]\fm$
is a proper ideal, the set common zeroes of $f\in\fm$ is nonempty.
\end{proof}

\begin{prp}
Let $R\subset V\times V$ be a saturated equivalence relation on $V$.
Then $R$ is Zariski closed in $V\times V$.
\end{prp}
\begin{proof}
Let $A=\C[V]^R\subset\C[V]$.
The relation $R$ is saturated, so $(x,y)\in R$ iff $f(x)=f(y)$ for any $f\in A$. In other words, $R$ is defined by the ideal of $\C[V\times V]=\C[V]\otimes\C[V]$ generated by the elements $f\otimes 1-1\otimes f$, $f\in A$.
\end{proof}

\subsubsection{An example}
Here is an example of a Zariski closed equivalence relation that is not
saturated. Let $V=\C$ and let $R$ be an equivalence relation on $\C$
generated by the condition $z\sim\frac{1}{z}$. The relation $R$ is Zariski closed
as it is given by the equation $(x-y)(xy-1)=0$. A polynomial $f$ satisfying the condition $f(z)=f(\frac{1}{z})$ is obviously constant,
so $\cA(R)=\C$ and $\cR(\cA(R))=V\times V$.

\subsection{}\label{fingr}
The finite group actions are detectable since $\C[V]$ is a finite
extension of the ring of invariants, and the claim follows from Going-Up Theorem. The aim of our work is
to prove detectability of certain equivalence relations appearing in Lie superalgebra theory.

\subsubsection{} 
\label{sss:superalgebras}
Let $\fg$ be a basic classical (finite-dimensional  
Kac-Moody) Lie superalgebra with a Cartan subalgebra $\fh$ and the Weyl group $W$. The Harish-Chandra projection identifies the center $Z\fg$ of the universal enveloping algebra of $\fg$ with a subalgebra of $S\fh^W$ that, by the works of A.~Segreev and V.~Kac, see~\cite{Sinv,Kac,MG0},
has form $\C[V]^R$ where $V=\fh^*$ and the equivalence relation $R$
is described as follows. Let $\Delta\subset\fh^*$ be the root system of 
$\fg$
with the bilinear form $\langle\cdot|\cdot\rangle$ and the set of isotropic roots $\Delta^\iso$. Then $R$ is generated by the Weyl group action, together with the condition
$$
(x,x+\lambda\alpha)\in R,\textrm{ for any }\alpha\in \Delta^\iso, 
\langle x|\alpha\rangle=0.
$$

If $\Delta^{\iso}$ is empty (for example, if $\fg$ is a Lie algebra),
the detectability of $R$ follows from the above observation~\ref{fingr}.
Detectability of  $R$ for basic classical $\fg$ has been recently established   by I.~Musson,
see~\cite{M}. 

\begin{thm}[I.~Musson]Let $\fg$ be a basic classical Lie superalgebra, $\fh$ a Cartan subalgebra, $Z\subset S\fh$ the image of the center of $U\fg$ 
under the Harish-Chandra projection. Then
\begin{itemize}
\item[1.] For any maximal ideal $\fm$ of $Z$ there exists $\lambda\in\fh^*$ so that $\fm=Z\cap\fm_\lambda$, where $\fm_\lambda$
is the ideal of polynomials vanishing at $\lambda$.
\item[2.] $Z\cap\fm_\lambda=Z\cap\fm_\mu$ if and only if $(\lambda,\mu)
\in R$ for the equivalence relation $R$ defined above.
\end{itemize}
\end{thm}
These two properties together are equivalent to detectability by \ref{lem:detsat}.
Musson's proof is based on a case-by-case analysis. The main
motivation of our work was to find a general argument for this sort
of questions.

\subsection{} This paper is based on two essential observations. The first one
is the idea that, for an equivalence relation $R$ on $V$, the quotient
space $V/R$, endowed with an appropriate structure of a ringed space
(we use in this paper the Kempf's language of ``spaces with functions'',
see~\cite{K}) is a meaningful geometric object containing more  
information than merely the ring of global regular functions $\C[V]^R$.  

The other point is the description of a beautiful family of 
equivalence relations on $V$ that we call Lagrangian equivalence
relations. Almost all examples of equivalence relations appearing in
Lie superalgebra theory are Lagrangian; it is easy to see that the
restriction of a linear Lagrangian relation to an appropriate subfactor remains Lagrangian; this is convenient for the inductive argument in proving detectability.

\subsection{}
Here is a baby example of an invariant ring of the type we will be 
dealing with. Let $V=\Span_\C(v,w)$ and the equivalence relation is 
generated by the condition $cv\sim 0\ \forall c\in\C$. The invariant subring of 
$\C[x,y]$ ($x$ and $y$ form the dual basis to $\{v,w\}$) is 
$A=\C+y(x,y)$. This is a non-Noetherian ring with the fraction field
$\C(x,y)$. If a maximal ideal $\fm\subset A$ does not contain $y$, it corresponds to a maximal ideal of $A_y=\C[x,y]_y$. 
The ring $A/(y)$ is a 
(non-Noetherian) local ring whose maximal ideal has square zero. Therefore, there is a unique maximal ideal $\fm\subset A$ containing $y$.

We see that the quotient $V/R$ has a natural stratification, consisting in this example of two strata, the open stratum isomorphic to 
$\bA^2\setminus\bA^1$ and the closed stratum being a point.

\subsection{}
The argument both in Musson's and in the present paper goes by 
induction on strata
and it is of crucial importance that the restriction of functions
to the next stratum is surjective. The proof of surjectivity,
in the context of finite dimensional Lie superalgebras, was given in 
the preprint~\cite{MG}, based on Sergeev-Veselov description of the
centers of Lie superalgebras. In the present paper the surjectivity
follows from elementary sheaf theory.

\subsection{Acknowledgements}
The authors enjoyed numerous visits to Weizmann Institute and to
UC Berkeley, whose pleasant atmosphere is gratefully acknowledged.
The results of this work
were presented in Berkeley, Rehovot and Prague. We thank 
I.~Entova-Aizenbud, V.~Mazorchuk, D.~Rumynin, A.~Sherman and Ilya
Zakharevich, for valuable discussions. M.G. was supported by ISF 1957/21 grant. V.S. was supported by NSF grant 2019694.

\section{Lagrangian relations}

\subsection{Linear relations}
Let $V$ be a complex vector space of dimension $n$ endowed with a nondegenerate symmetric bilinear form $v,w\mapsto\langle v|w\rangle$.

We are going to study a class of equivalence relations on $V$ that we call Lagrangian equivalence relations.

We define a bilinear form $B$ ob $V\times V$ by the formula
$$
B((v,w),(v',w'))=\langle v|w\rangle-\langle v'|w'\rangle.
$$

\begin{dfn} A linear relation on $V$ is a vector subspace $L\subset V\times V$. A linear relation $L$ is called 
{\em isotropic}  if $B|_L=0$ and  is called a {\em linear Lagrangian relation} if it is isotropic of maximal dimension $\dim(V)$.
\end{dfn}
 Since $B$ is symmetric, $L$ is isotropic iff any $(x,y)\in L$ satisfies
  $\langle x|x\rangle=\langle y|y\rangle$. 

\begin{exm}For an endomorphism $\alpha:V\to V$ the graph $\Gamma_\alpha=
\{(v,\alpha(v))\}$ is a linear Lagrangian  relation iff $\alpha$ is an isometry.
\end{exm}

\subsection{Composition}

Given two relations $L,L'\subset V\times V$, their composition
$L'\circ L$ is defined as the collection of pairs $(x,z)$ for which there
exists $y$ with $(x,y)\in L$ and $(y,z)\in L'$. The composition of
linear relations is linear. The definition of composition is chosen so that, for $\alpha,\beta\in\End(V)$, 
$\Gamma_\alpha\circ\Gamma_\beta=\Gamma_{\alpha\circ\beta}$.

\begin{lem}If $L,L'$ are isotropic relations, their composition is 
also isotropic.
\end{lem}
\begin{proof}
If $(x,y)\in L$ then $\langle x|x\rangle=\langle y|y\rangle$; if
$(y,z)\in L'$ then $\langle y|y\rangle=\langle z|z\rangle$. Thus,
  $\langle x|x\rangle=\langle z|z\rangle$ for any $(x,z)\in L'\circ L$.
\end{proof}
If $L$ is a linear relation, the relation $L^{-1}:=\{(x,y)|(y,x)\in L\}$ is also
linear. It is isotropic iff $L$ is isotropic.
For a linear relation $L$ we denote by $p_1,p_2$ the projections of
$L\subset V\times V$ to the two copies of $V$. We further denote
$K_i=\Ker(p_i)\subset L$.

\begin{lem}Let $L$ be a linear Lagrangian relation. Then $\dim K_1=\dim K_2$.
\end{lem}
\begin{proof}
Denote $k_i=\dim K_i$.
Let us calculate the dimension of $L\circ L^{-1}$. The fiber product
$L\times_VL^{-1}$ has dimension $2k_2+(n-k_2)$ and the kernel of
its projection to $V\times V$ is $K_1\times_VK_1$ that is isomorphic
to $K_1$ as the projection $p_2:K_1\to V$ is injective. Therefore,
$\dim L\circ L^{-1}=n+k_2-k_1$. The composition is isotropic, so
$k_2\leq k_1$. By symmetricity we get the equality.
\end{proof}

\begin{dfn}
Let $L$ be a linear Lagrangian relation. The number $\dim K_1=\dim K_2$ is
called {\sl the atypicality} of $L$.
\end{dfn}
For example, a linear Lagrangian relation of atypicality $0$ is a graph of an isometry on $V$.
We denote by $a(L)$ the atypicality of $L$.

\begin{lem}
\label{lem:ker-ort-im}
Let $L$ be a linear Lagrangian relation. Then $p_1(L)$ is the orthogonal complement of $p_1(K_2)$. In particular, $p_1(L)$ and $p_2(L)$ are coisotropic subspaces of $V$.
\end{lem}
\begin{proof}
An element of $K_2$ is $(x,0)\in L$. If $y\in p_1(L)$, there exists
$(y,z)\in L$. Since $B((x,0),(y,z))=0$, we have $\langle x|y\rangle=0$.
Finally, since $p_1|_{K_2}$ is injective and $\dim K_2+\dim p_1(L)=n$,
we deduce that $p_1(K_2)$ is the orthogonal complement of $p_1(L)$. 
\end{proof}

Finally, here is the main result of this subsection.

\begin{prp}
\label{prp:monoid}
The set of linear Lagrangian relations on $V$ is a monoid with respect to
the composition.
\end{prp}
\begin{proof} Let $L,L'$ be linear Lagrangian relations, $k=a(L)$ and 
$k'=a(L')$. The composition
$L'\circ L$ is isotropic. Let us calculate the dimension of $L'\circ L$. We have
$$
L'\circ L=\mathrm{Im}(L\times_VL'\to V\times V),
$$
so $\dim(L'\circ L)=k+k'+\dim(p_2(L)\cap p_1(L'))-
\dim(p_2(K_1)\cap p_1(K'_2))$.
Since $p_2(K_1)=p_2(L)^\perp$ and $p_1(K_2')=p_1(L')^\perp$, 
$p_2(K_1)\cap p_1(K_2')=(p_2(L)+p_1(L'))^\perp$ and therefore
$$
\dim(L'\circ L)=k+k'+n-\dim(p_2(K_1)+p_1(K'_2))-\dim(p_2(K_1)\cap p_1(K'_2))=n.
$$
 
\end{proof}

\begin{lem}
\label{lem:atypicality-ineq}
One has
$$
\max(a(L),a(L'))\leq a(L'\circ L)\leq a(L)+a(L').
$$
\end{lem}
\begin{proof}
Obviously $p_1(L\circ L')\subset p_1(L)$ and
$p_2(L\circ L')\subset p_2(L')$. This immediately implies that
$a(L'\circ L)\geq\max(a(L),a(L'))$.
To get the second inequality, let us estimate the dimension of $p_1(L'\circ L)$.
Denote $P=p_2(L)\cap p_1(L')$. Obviously, $\dim(P)\geq n-a(L)-a(L')$,
so $\dim p_2^{-1}(P)\geq n-a(L')$. Thus, $\dim p_1(L'\circ L)=\dim p_1(p_2^{-1}(P))\geq
n-a(L)-a(L')$. This implies that $a(L'\circ L)\leq a(L)+a(L')$.
\end{proof}

\subsection{More on linear Lagrangian relations}

Linear Lagrangian relations of atypicality $0$ are defined by isometries of $V$.
A general linear Lagrangian relation can also be described in terms of
isometries.

\begin{lem}
\label{lem:L-isom}
Let $V_0$ and $V'_0$ be two coisotropic subspaces of $V$ of the same dimension. Let $V_1=V_0^\perp$ and $V_1'=(V_0')^\perp$. Then there is a one-to-one correspondence between linear Lagrangian relations $L$ with $p_1(L)=V_0$ and $p_2(L)=V_0'$ and isometries $\alpha:V_0/V_1\to V_0'/V_1'$.
\end{lem}
\begin{proof}
We look at $L\subset V_0\times V_0'$ as at a relation between $V_0$ and $V_0'$. We claim that the same relation is given by an isometry 
$\alpha$. For $v\in V_0$ choose $(v,v')\in L$ and send $v\in V_0$
to $v'$. This map is uniquely defined modulo $V_1$ and $V_1'$ and it
preserves the norm as $0=B((v,v'),(v,v'))=\langle v|v\rangle-\langle v'|v'\rangle$.
\end{proof}

\begin{dfn}
A linear Lagrangian relation $E$ is called {\sl idempotent} if $E\circ E=E$.
\end{dfn}

\begin{exm}
Let $V_0$ be a coisotropic subspace of $V$ and let $V_1=V_0^\perp$.
Then 
$$
E_{V_0}:=\{(v,w)\in V\times V|v\in V_0,\ w\in V_1\}
$$
is an idempotent linear Lagrangian relation.
\end{exm}
We will prove soon that any idempotent linear Lagrangian relation is of the above form.

\begin{lem}
\label{lem:L-dec}
Any linear Lagrangian relation $L$ can be presented as
$$
L=\Gamma_\gamma\circ E_{V_0},
$$
where $\gamma$ is an isometry of $V$ and $V_0=p_1(L)$.
\end{lem}
\begin{proof}
Let $L$ be Lagrangian and $\gamma\in O(V)$.
Obviously  $p_1(\Gamma_\gamma\circ L)=p_1(L)$ and 
$p_2(\Gamma_\gamma\circ L)=\gamma(p_2(L))$.

Let $V_0:=p_1(L)$ and $V_0':=p_2(L)$. 
Choose
$\gamma\in O(V)$ such that $V_0'=\gamma(V_0)$. 
Then $V_0=p_1(\Gamma_{\gamma^{-1}}\circ L)=
p_2(\Gamma_{\gamma^{-1}}\circ L)$. So by the previous description it is defined by an isometry of $V_0/V_0^\perp$. Any such isometry can be lifted to an isometry of $V$
preserving $V_0$. This implies the claim.
\end{proof}

We are now ready to deduce the characterization of idempotent Lagrangian
relations.

\begin{prp}Let $E$ be an idempotent linear Lagrangian relation. Then $E=E_{V_0}$
for $V_0=p_1(E)$.
\end{prp}
\begin{proof}
By Lemma~\ref{lem:L-dec}, $E=\Gamma_\gamma\circ E_{V_0}$ where 
$V_0=p_1(E)$ and $\gamma\in O(V)$.
Since $E$ is idempotent,
$$
E_{V_0}\circ\Gamma_\gamma\circ E_{V_0}=E_{V_0}.
$$ 
We will deduce first of all that $\gamma(V_0)=V_0$;
the rest will be very easy.

We have $E_{V_0}=\{(v,v+w)|v\in V_0,\ w\in V_1\}$ where $V_1=V_0^\perp$.
Then 
$$
E_{V_0}\circ\Gamma_\gamma\circ E_{V_0}=\{(v,\gamma(v+w)+w')|
v\in V_0,\ w,w'\in V_1,\ \gamma(v+w)\in V_0\}.
$$
Since this space coincides with $E_{V_0}$, we have  

$$
\forall v\in V_0\ \quad\exists w\in V_1:\gamma(v+w)\in v+V_1. 
$$

Choose elements $v_1,\ldots,v_k$ in $V_0$ whose images in $V_0/V_1$
form a basis. Choose $w_i,w_i'\in V_1$ such that
$$
\gamma(v_i+w_i)=v_i+w'_i.
$$
Replacing $v_i$ with $v_i+w_i$ and $w_i'$ with $w_i$ we simplify
the formulas and get
$$
\gamma(v_i)=v_i+w_i.
$$
Since $\gamma$ is an isometry and $V_1=V_0^\perp=\Span\{v_1,\ldots,v_k\}^\perp$, we deduce that $\gamma(V_1)=V_1$. This implies that 
$\gamma(V_0)=V_0$.

As a result, $E=\Gamma_\gamma\circ E_{V_0}$ satisfies $p_1(E)=p_2(E)=V_0$.
By Lemma~\ref{lem:L-isom} $E$ is described by an isometry $\alpha$ of $V_0/V_1$. Then
$E\circ E$ is described by the isometry $\alpha^2$, so $\alpha$ is idempotent and therefore $\alpha=1$.
\end{proof}

Note the following
\begin{lem}
\label{lem:ll_1}
For any linear Lagrangian relation $L$ the composition $L^{-1}\circ L$ is
idempotent.
\end{lem}
\begin{proof}
Let $E=L^{-1}\circ L$ and let $V_0=p_1(L)$. Then $V_0=p_1(E)$ and 
$E$ contains the diagonal $\Delta_{V_0}=\{(v,v)|v\in V_0\}$. This leads
to the direct decompositions
$E=\Delta_{V_0}\oplus K_1=\Delta_{V_0}\oplus K_2$ that implies
$E=E_{V_0}$.
\end{proof}

\subsection{Lagrangian equivalence relations}

\begin{dfn}
\label{dfn:lagrangian-relation}
An equivalence relation $R\subset V\times V$ that is a union of linear Lagrangian relations, $R=\cup_{\alpha\in A}L^\alpha$, is called a Lagrangian equivalence relation. 
\end{dfn}

Note: a Lagrangian equivalence relation is not a linear Lagrangian relation!

In this paper we are mostly interested in 
Lagrangian equivalence relations having a finite number of components.

In other words, a Lagrangian equivalence relation is defined by a 
collection $\{L^\alpha\}$ of linear Lagrangian relations, with unit (the diagonal $V\subset V\times V$), closed under the composition and taking inverse. 

\begin{dfn}
Given a Lagrangian equivalence relation $R=\cup_{\alpha\in A}L^\alpha$,
a subspace $p_1(L^\alpha)\subset V$   is called a
{\sl special coisotropic subspace}
of $V$.
\end{dfn}
If $V_0$ is a special coisotropic subspace with $V_1=V_0^\perp$,
$R$ contains an idempotent component $E_{V_0}:=\{(v,v+w)|v\in V_0,w\in V_1\}$ by Lemma~\ref{lem:ll_1}.
Special coisotropic subspaces are in one-to-one correspondence with the idempotent components of $R$.  

\subsubsection{The Weyl group}

Let $R=\cup_{\alpha\in A}L^\alpha$ be a Lagrangian equivalence relation. The components of
atypicality $0$ form a group that will be called the Weyl group $W$ of $R$.
Any $s\in W$ is an isometry of $V$ such that the graph $\Gamma_s$ is a
component of $R$. Equivalently, $W$ is the group of invertible elements 
of the monoid $R$. We think of $W$ as a group of isometries of $V$
rather than  the group of invertible components of $R$.  

The group $W$ acts on $\{L^\alpha,\alpha\in A\}$ by composing on the right and on the left.
We have some obvious identities
\begin{itemize}
\item[1.] $p_1(\Gamma_s\circ L^\alpha)=p_1(L^\alpha)$, \
$p_2(\Gamma_s\circ L^\alpha)=s(p_2(L^\alpha))$.
\item[2.] $E_{s(V_0)}=\Gamma_s\circ E_{V_0}\circ\Gamma_{s^{-1}}$.
\end{itemize}

\subsection{Reduction}

Let $R$ be a Lagrangian equivalence relation on $V$ having a finite number of components and let $V_0$ be a special
coisotropic subspace. Set $V_1=V_0^\perp$. The main result of this subsection, Proposition~\ref{prp:red}, asserts that the restriction
of $R$ to $V_0$ induces
a Lagrangian equivalence relation on $V_0/V_1$ which we will denote by $R|_{V_0/V_1}$.

\subsubsection{}

Note that $(v,v+w)\in R$ for any $v\in V_0$ and $w\in V_1$, therefore
$R\cap(V_0\times V_0)$ is invariant with respect to the shifts
by $V_1\times V_1$. 
The intersection $R\cap(V_0\times V_0)=\cup_{\alpha\in A} L^\alpha\cap(V_0\times V_0)$ is a finite union of vector subspaces of $V_0\times V_0$; since it is invariant with respect to $V_1\times V_1$,
it is the union of the $V_1\times V_1$-invariant components
among  $L^\alpha\cap(V_0\times V_0)$. 

\begin{lem}Let $L$ be a linear Lagrangian relation and let $V_0$ be a coisotropic subspace of $V$ with $V_1=V_0^\perp$. Then the intersection 
$L\cap(V_0\times V_0)$ is $V_1\times V_1$-invariant iff  
$L\subset V_0\times V_0$.
\end{lem}
\begin{proof}
If $L\cap(V_0\times V_0)$ is $V_1\times V_1$-invariant then
$L\supset V_1\times V_1$, so $L=L^{\perp}\subset(V_1\times V_1)^\perp=V_0\times V_0$. 

Vice versa, if $L\subset V_0\times V_0$ then
$L\supset V_1\times V_1$. Therefore,
$L\cap(V_0\times V_0)=L$ is $V_1\times V_1$-invariant.

\end{proof}

\begin{prp}
\label{prp:red}
Let $R=\cup_{\alpha\in A}L^\alpha$ be a Lagrangian equivalence relation
and let $V_0\subset V$ be a special coisotropic subspace with $V_1=V_0^\perp$. Then the restriction of $R$ to $V_0$ induces an
equivalence relation  on $V_0/V_1$ which is also Lagrangian. Its components are of the form $L^\alpha/(V_1\times V_1)$ for 
$\alpha\in A'$ where
$$
A'=\{\alpha\mid L^\alpha\subset V_0\times V_0\}=
\{\alpha\mid E_{V_0}\circ L^\alpha\circ E_{V_0}=L^\alpha\}.
$$
\end{prp}
\begin{proof}
Only the second equality in the description of $A'$ requires explanation.
For any $\alpha\in A$ the intersection $L^\alpha\cap(V_0\times V_0)$ lies 
in (at least one) $L^{\alpha'}$ satisfying the condition
$L^{\alpha'}\subset V_0\times V_0$.  It is easy to present one of them:
this is 
$L^{\alpha'}:=E_{V_0}\circ L^\alpha\circ E_{V_0}$. So, $\alpha\in A'$ iff 
$L^\alpha=E_{V_0}\circ L^\alpha\circ E_{V_0}$.
\end{proof}

\subsection{Discriminant}
Let $R$ be a Lagrangian equivalence relation on $V$. The discriminant
of $R$, $\Delta_R$, is defined as the union of all proper
special coisotropic subspaces of $V$. Thus, if $R=\cup L^\alpha$,
$$
\Delta_R=\bigcup_{a(L^\alpha)>0}p_1(L^\alpha).
$$

\begin{lem}
\label{lem:birational}
Let $W$ be the Weyl group of a Largangian equivalence relation $R$
with finitely many components.
Then $\C[V]^R\subset\C[V]^W$ and 
the rings $\C[V]^R$ and $\C[V]^W$ have the same field of fractions. 
\end{lem}
\begin{proof}
The discriminant $\Delta_R$ is Zariski closed in $V$, so there is a
$W$-invariant polynomial $T$ vanising on $\Delta_R$: such is,
for instance, $T=\prod_{s\in W}s(f)$ where $f$ vanishes on 
$\Delta_R$. This implies that $T\in\C[V]^R$ and moreover $T\C[V]^W\subset \C[V]^R$.
Therefore, the localizations of  $\C[V]^R$ and of $\C[V]^W$ with respect
to $T$, coincide.
This proves the lemma.
\end{proof}

\subsection{Regular Lagrangian equivalence relations}

\begin{dfn}A Lagrangian equivalence relation $R$ on $V$ with finitely many components is called
$1$-{\sl regular} if either $\Delta_R=\emptyset$ (all components $L^\alpha$ have atypicality $0$, that is, $R=W$), or $\Delta_R=WV_0$ where $V_0$ is a special coisotropic subspace of codimension $1$.
\end{dfn}

If a Lagrangian equivalence relation $R$ is $1$-regular with a
codimension $1$ special coisotropic subspace $V_0\subset V$, one can
easily describe the Weyl group of the reduction of $R$ to $V_0/V_1$. 
\begin{prp}
Let $R$ be $1$-regular. Then the Weyl group $W'$ of the
equivalence relation $R':=R|_{V_0/V_1}$ is 
the quotient of the stabilizer $\{s\in W|s(V_0)=V_0\}$
modulo the normal subgroup of $s\in W$ inducing the identity on $V_0/V_1$.
\end{prp}
\begin{proof}
We are looking for $L^\alpha$, $\alpha\in A'$, satisfying the condition $p_1(L^\alpha)+V_1=V_0$. Since $p_1(L^\alpha)\subset V_0$ is also coisotropic, 
$p_1(L^\alpha)\supset p_1(L^\alpha)^\perp\supset V_1$ so in fact
$p_1(L^\alpha)=V_0$. By 1-regularity, $L^\alpha=E_{V_0}\circ\Gamma_s$ for some $s\in W$ and the condition $p_2(L^\alpha)=V$
implies that $s$ stabilizes $V_0$.
\end{proof}
\begin{dfn}
A Lagrangian equivalence relation $R$ is called {\sl regular} if for any special
coisotropic subspace $V_0$ with $V_1=V_0^\perp$ the equivalence
relation $R'=R|_{V_0/V_1}$ is $1$-regular.
\end{dfn}

Obviously, if $R$ is regular, the reduction $R'=R|_{V_0/V_1}$ is
also regular for any special coisotropic subspace $V_0$ of $V$.

\subsection{Product of equivalence relations}
Let $V,V'$ be two vector spaces with equivalence relations $R\subset V\times V$ and $R'\subset V'\times V'$. The equivalence relation
$R\times R'$ on $V\times V'$ is defined in an obvious way:
$((v,v'),(w,w'))\in R\times R'$ iff $(v,w)\in R$ and $(v',w')\in R'$.
The isomorphism $\C[V\times V']=\C[V]\otimes\C[V']$ defines a homomorphism 
$$
\mu:\C[V]^R\otimes\C[V']^{R'}\to\C[V\times V']^{R\times R'}.
$$
\begin{lem}
The map $\mu$ is bijective.
\end{lem}
\begin{proof}
The map $\mu$ is obviously injective. Let us verify its surjectivity.
Given an invariant polynomial $f\in\C[V\times V']^{R\times R'}$, let
$$
f=\sum p_i\otimes q_i,
$$
with $p_i\in\C[V]$ and $q_i\in\C[V']$, be a shortest presentation of $f$.
This means, in particular, that $p_i$ are linearly independent in $\C[V]$
and $q_i$ are linearly independent in $\C[V']$.
Note that for any $y\in V'$ the function $f(-,y)=\sum p_iq_i(y)$ belongs to $\C[V]^R$. By Lemma~\ref{lem:lin} below we deduce that there are 
$y_i\in V'$ such that the matrix $\left(q_i(y_j)\right)$ is invertible. This 
implies that all the $p_i$ belong to $\C[V]^R$. In a similar way
we deduce that all the $q_i$ belong to $\C[V']^{R'}$, so that $f\in
\C[V]^R\otimes\C[V']^{R'}$.
\end{proof}

The following lemma must be well-known.
\begin{lem}
\label{lem:lin}
Let $f_1,\ldots,f_n:X\to\C$ be $n$ linear independent functions
on an infinite set $X$. Then there exist $x_1,\ldots,x_n\in X$ such that
the matrix $(f_i(x_j))$ is invertible.
\end{lem}
\begin{proof}
Let $k$ be the maximal number for  which there exist $x_1,\ldots,x_k$
that yield a rank $k$ matrix $(f_i(x_j))$. We have to verify that $k=n$.
Assume $k<n$. There exists a nonvanishing $k$-minor of the matrix $(f_i(x_j))$; without loss of generality we can assume that this is the
minor formed by the functions $f_1,\ldots,f_k$. We will now prove that
the functions $f_1,\ldots,f_{k+1}$ are linearly dependent. Look at the
$(k+1)\times(k+1)$-matrix formed by the values of the functions
$f_1,\ldots,f_{k+1}$ at the points  $x_1,\ldots,x_k,x$ of $X$. We get a square matrix that is not
invertible for all values of $x\in X$. This means that its determinant
vanishes. Decomposition along the last column yields a linear dependence 
of $f_i(x)$. 
\end{proof}

Product of equivalence relations preserves detectability.

\begin{lem}
\label{lem:product-det}
Let $R$ and $R'$ be detectable equivalence relations on $V$ and $V'$, respectively. Then $R\times R'$ is a detectable equivalence relation on $V\times V'$.
\end{lem}
\begin{proof}Denote $A=\C[V]^R$, $A'=\C[V']^{R'}$. Let $\fm\subset A\otimes A'$ be a maximal ideal. 

We have $\dim_{\C}(A\otimes A')/\fm\leq\dim_{\C}\C[V]\otimes\C[V']
=\aleph_0$, whereas $\dim_{\C}\C(x)>\aleph_0$, 
so $(A\otimes A')/\fm=\C$. This implies that the set of maximal ideals of $A\otimes A'$
coincides with $\Hom(A\otimes A',\C)=\Hom(A,\C)\times\Hom(A',\C)$. 
The rest of the claim is obvious.
\end{proof}

\subsection{Semiregularity}

\begin{dfn}A Lagrangian equivalence relation $R$ on a vector space $V$
is called $1$-semiregular if $V$ admits an orthogonal decomposition
$V=V^1\times\ldots\times V^k$, $R=R^1\times\ldots\times R^k$ such that
all $R_i$ are $1$-regular Lagrangian equivalence relations on $V_i$.
\end{dfn}

\begin{dfn}
A Lagrangian equivalence relation $R$ on a vector space $V$ is called
semiregular if for any special coisotropic subspace $V_0$ of $V$ with $V_1=V_0^\perp$ the equivalence relation $R'=R|_{V_0/V_1}$ is
$1$-semiregular.
\end{dfn}

The main result of our paper is

\begin{THM}
\label{thm:semireg-det}
Any semiregular  Lagrangian equivalence relation $R$ is
detectable: there is a one-to-one correspondence between the
maximal ideals of $\C[V]^R$ and the set $V/R$ of equivalence classes
of $V$ modulo $R$.
\end{THM}

\subsection{Proof of Theorem~\ref{thm:semireg-det}} We will now present the proof of Theorem~\ref{thm:semireg-det}, modulo a surjectivity
result~\ref{prp:surj}.
Using Lemma~\ref{lem:product-det}, we can assume that
$R$ is $1$-regular (and semiregular).
Let $V_0$ be a special coisotropic subspace of $V$ of codimension
$1$. The discriminant $\Delta_R$ has form $\cup_{s\in W}s(V_0)$.

We denote by $T\in\C[V]$ the $W$-invariant polynomial of a minimal degree
that vanishes at $V_0$. Obviously $T\in\C[V]^R$.
Set $R':=R|_{V_0/V_1}$.
Our proof of Theorem~\ref{thm:semireg-det}
is based on the following result.
\begin{prp}
\label{prp:surj}
The restriction map
$$
i^*:\C[V]^R\to\C[V']^{R'}
$$
is surjective, with the kernel $\rad(T)$.
\end{prp}
The proof of Proposition~\ref{prp:surj} is given in Subsection~\ref{ss:surj}. 
\subsubsection{}
We complete the proof of Theorem~\ref{thm:semireg-det} by induction. A maximal ideal $\fm\subset\C[V]^R$ is called {\sl typical} if $T\not\in\fm$. Typical maximal ideals correspond
to the maximal ideals of the localization $\C[V]^R_T=\C[V]^W_T$, see
Lemma~\ref{lem:birational}. This implies that the maximal ideals of
$\C[V]$ over a typical $\fm$  form one $W$-orbit
which is an equivalence class of $R$.

Let us now deal with an atypical maximal ideal $\fm\subset \C[V]^R$.
Then $T\in\fm$, so $\fm$ contains $\rad(T)$  and thus by
Proposition~\ref{prp:surj}   it defines a maximal ideal $\fm'$ of $\C[V']^{R'}$. By induction hypothesis, the zeros of $\fm'$ form a 
single equivalence class $X'\subset V'$ of $R'$. Since $R$ is 1-regular, 
the set of zeros of $\fm$ in $V$ is the $W$-orbit of the preimage in 
$V_0$ of $X'$.

 \section{Spaces with functions}
We will use the formalism of ``spaces with functions''~\cite{K}, 1.1, to 
describe quotients of $V$ by an  equivalence 
relation.

\

Recall that a space with functions (over the field $\C$
in this note) is a pair $(X,\cO_X)$ where $X$ is a topological space,
$\cO_X$ is a subsheaf of the sheaf of functions $U\mapsto\Map(U,\C)$
such that
\begin{itemize}
\item if $U$ is an open subset of $X$ and $f\in\cO_X(U)$ then the set
$\{x\in U|f(x)=0\}$ is closed in $U$.
\item $\cO_X(U)$ is a subring of $\Map(U,\C)$. If $f\in\cO_X(U)$ does not vanish at $U$ then $\frac{1}{f}\in\cO_X(U)$.
\end{itemize}

A morphism of spaces with functions $f:(X,\cO_X)\to(Y,\cO_Y)$ is a
continuous map $f:X\to Y$ such that for any $g\in\cO_Y(V)$ the map
$f^*(g):f^{-1}(V)\to\C$ is in $\cO_X(f^{-1}(V))$.

The category of spaces with functions is very convenient to work with reduced algebraic varieties of 
finite type over an algebraically closed field. It appears that it is also convenient for studying quotients by equivalence relations.

\subsection{Factor space}

Given a space with functions $(X,\cO_X)$ and an equivalence relation $R$ on $X$,
one defines a new space with functions $(X/R,\cO_{X/R})$, universal
among morphisms $f:(X,\cO_X)\to (Y,\cO_Y)$ for which $f(x)=f(y)$ for any
$(x,y)\in R$. As a topological space, $X/R$ is precisely the quotient
of $X$  modulo
the equivalence relation (the set $X/R$ with the weakest topology
making the canonical map $\pi:X\to X/R$ continuous). For an open subset $U\subset X/R$ the ring $\cO_{X/R}(U)$ consists of
the functions $f:U\to\C$ for which the composition 
$\pi^{-1}(U)\stackrel{\pi}{\to} U\stackrel{f}{\to}\C$ belongs to $\cO_X(\pi^{-1}(U))$. 

\subsection{Subspace}

Given a space with functions $(X,\cO_X)$ and a subset $Z\subset X$,
there is a canonical structure of a space with functions on $Z$ that is
universal among the morphisms $f:(Y,\cO_Y)\to (X,\cO_X)$ for which $f(Y)\subset X$. As a topological space, $Z$ acquires the induced topology.
For an open subset $V\subset Z$, a function $f:V\to \C$ is regular if
for any $x\in V$ there exists open $U\subset X$ with $x\in U$ and a regular function $\tilde f$ on $U$ such that 
$f_{|U\cap V}=\tilde f_{|U\cap V}$. A map of spaces with functions
$f:(Z,\cO_Z)\to(X,\cO_X)$ is called an embedding if $f$ establishes a
homeomorphism of $Z$ with $f(Z)\subset X$ and identifies the sheaf 
$\cO_Z$ with the sheaf on $f(Z)$ described above.

\begin{lem} \label{lem:surj}
  Let $(X,\cO_X)$ be a space with functions, $Z\subset X$ closed
and let $i:(Z,\cO_Z)\to(X,\cO_X)$ be a closed embedding of 
spaces with functions.
The canonical map of sheaves $\cO_X\to i_*(\cO_Z)$ is surjective whose kernel $I$ is the ideal of regular functions vanishing at $Z$.
\end{lem}
\begin{proof}
The canonical map in question is defined by the collection of maps
$$
\cO_X(U)\to i_*\cO_Z(U)=\cO_Z(Z\cap U)
$$
that assign to any regular function on $U$ its restriction to $Z\cap U$.
The kernel is defined by the formula
$$
I(U)=\{f\in\cO_X(U)|f_{|U\cap Z}=0\}.
$$
It remains to verify that the map of sheaves $\cO_X\to i_*\cO_Z$
is surjective. It suffices to prove that for any $x\in X$ the map
of stalks $\cO_{x,X}\to(i_*\cO_Z)_x$ is surjective. If $x\not\in Z$ then
$(i_*\cO_Z)_x=0$. Assume $x\in Z$. Given a stalk
$(U,g\in\cO_Z(U\cap Z))$, there exists $U'$ containing $x$ and a function
$\tilde g\in\cO_X(U')$ such that $\tilde g_{|U'\cap U\cap Z}=g_{|U'\cap U\cap Z}$. The stalk at $x$ defined by $\tilde g$ lifts $(U,g)$.
\end{proof}

\subsection{}
We will need the following elementary result.

\begin{lem}
\label{lem:reducible}
Let $X$ be an affine algebraic variety over an algebraically closed field $k$ with irreducible components $X_1,\ldots,X_n$. A function
$f:X\to k$ is regular iff its restriction to each component $X_i$ 
is regular.
\end{lem}
\begin{proof}
Follows from R. Godement's Th\'eor\`eme II.1.3.1~\cite{G}.
The irreducible components of a variety form a closed locally finite covering, so that the \v{C}ech complex corresponding to this covering is exact.
\end{proof}

\subsection{Proof of Proposition~\ref{prp:surj}}
\label{ss:surj}

\subsubsection{}

We denote by $X=V/R$ the quotient considered as a space with functions.
Let $Z\subset V/R$ be the image of $V_0$ in $X=V/R$. By Lemma
\ref{lem:surj} the map $\cO_X\to i_*\cO_Z$ is surjective with the kernel
$I\subset\cO_X$ generated by the functions in $\cO_X$ vanishing at $Z$.

Let $W$ be the Weyl group of $R$.
We denote by $p:V/W\to V/R$ the canonical projection. Note that $V/W$
is the affine variety defined by the ring $\C[V]^W$ of regular functions.

The canonical map $\cO_X\to p_*\cO_{V/W}$ gives rise to the injective map
$\C[V]^R\to\C[V]^W$ discussed in \ref{lem:birational}.
Recall that $T\in\C[V]$ is the $W$-invariant polynomial of a minimal degree
that vanishes at $V_0$.
\begin{lem}
The kernel of the canonical map $\cO_X\to i_*\cO_Z$ is the ideal $I$
(globally) generated by $T\C[V]^W$.
\end{lem}
\begin{proof}
Let $U$ be an open $R$-invariant subset of $V$.
Any function $f\in\cO_V(U)^R\subset\cO_V(U)^W$ vanishing on $\Delta_R=WV_0$ is divisible by $T$, $f=Tg$, 
$g\in\cO_V(U)$. Obviously, since $f\in\cO_V(U)^W$, then $g\in\cO_V(U)^W$. On the
other hand, if $g\in\cO_V(U)^W$ then $Tg\in\cO_V(U)^R$.
\end{proof}

\begin{crl}
\label{crl:surj1}
The restriction to $Z$ defines a surjective ring homomorphism
$$\cO_X(X)\to\cO_Z(Z).$$
\end{crl}
\begin{proof}
We have a short exact sequence of sheaves on $X=V/R$
$$
0\to p_*(\cO_{V/W})\to \cO_X\to i_*\cO_Z\to 0,
$$
where the first map is the multiplication by $T\in\C[V]^R$. This short
exact sequence gives rise to the exact sequence
$$
0\to\C[V]^W\stackrel{T}{\to}\C[V]^R\to\cO_Z(Z)\to H^1(X,p_*(\cO_{V/W})).
$$
The exact sequence of the low-degree terms of the spectral sequence
$$
E_2^{pq}=H^p(X,R^qp_*(\cO_{V/W}))\Longrightarrow H^n(V/W,\cO_{V/W})
$$
gives an embedding of $E_2^{10}=H^1(X,p_*\cO_{V/W})$ into
$E_\infty^1=H^1(V/W,\cO_{V/W})=0$ (as $V/W$ is affine), so 
$H^1(X,p_*(\cO_{V/W}))$ vanishes. This proves the claim.
\end{proof}

By~\ref{crl:surj1} the map
\begin{equation}
\label{eq:surj}
\C[V]^R=\cO_X(X)\to\cO_Z(Z)=\C[WV_0]^R
\end{equation}
is surjective, where $Z\subset X$ is the image of $V_0\subset V$.
We will now give another presentation of the target of this map.

Recall that the equivalence relation $R|_{V_0}$ induces an equivalence
relation on the quotient $V_0/V_1$. We will denote it by $R'$.

We have a restriction map 

\begin{equation}
\label{eq:surj2}
\C[V]^R\to\C[V_0]^{R|_{V_0}}=\C[V_0/V_1]^{R'}
\end{equation}
that factors through (\ref{eq:surj}). This defines a canonical map
\begin{equation}
\label{eq:iso}
r:\C[WV_0]^R\to \C[V_0/V_1]^{R'}.
\end{equation}
\begin{lem}
The map (\ref{eq:iso})
is a bijection.
\end{lem} 
\begin{proof}
It is easy to see that $r$ is always injective: if an $R$-invariant
function vanishes on $V_0$, it is zero. We are going to prove that
$r$ is surjective. 

We claim that any $R'$-invariant $\C$-valued function on $V_0/V_1$ lifts
to an $R$-invariant function on $WV_0$. Indeed, any $x\in WV_0$ can
be presented as $sy$ where $y\in V_0$. If $x=sy=tz$ for $s,t\in W$ and
$y,z\in V_0$, we have $z=t^{-1}sy$, so $f(z)=f(y)$. This proves that
we can define $\tilde f(z)=f(y)$ where $z=sy$ and this definition will 
be independent of the presentation $z=sy$. This defines an $R$-invariant
function $\tilde f:WV_0\to\C$ extending $f:V_0\to\C$. It remains to prove that the function $\tilde f$ is regular. This follows from
Lemma~\ref{lem:reducible}.
\end{proof}

\subsubsection{}
Here is the last claim of Proposition~\ref{prp:surj}.
The kernel of $i^*:\C[V]^R\to\C[V']^{R'}$ is $T\C[V]^W$. This is a radical ideal as $\C[V']^{R'}$ is reduced. On the other hand,
$(T\C[V]^W)^2=TT\C[V]^W\subset T\C[V]^R=(T)$. Therefore,
$T\C[V]^W=\rad(T)$.

\section{Example: weak generalized root systems}

In this section we apply Theorem~\ref{thm:semireg-det} to prove detectability of the equivalence
  relations defined in \ref{sss:superalgebras}. The rings defined by these relations are isomorphic to the centers of universal enveloping algebras of finite-dimensional Kac-Moody superalgebras.

  In fact we consider a larger class of
  equivalence relations coming from weak generalized root systems
  in the sense of 
  Serganova, see~\cite{S}, Sect. 7. We check that these relations are Lagrangian
  and semiregular,  see~\ref{ss:semiregular}. Therefore they are detectable.
  The case of Kac-Moody superalgebras is related to  the generalized root systems that are special cases of weak generalized roots systems.

In this section
 $V$ is a finite-dimensional complex vector space with a nondegenerate
symmetric bilinear form.

\subsection{Weak generalized root systems}

Recall a generalization of the notion of root system due to V.~Serganova, \cite{S}.

\begin{dfn}
\label{dfn:WGRS}
 A finite subset $\Delta\subset V\setminus\{0\}$ is called a weak generalized root system (WGRS) if
\begin{itemize}
\item[1.] $\Delta=-\Delta$.
\item[2.] If, for $\alpha\in\Delta$, $\langle\alpha|\alpha\rangle\ne 0$,
then for any $\beta\in\Delta$ one has $k_{\alpha,\beta}:=\frac{2\langle\alpha|\beta\rangle}{\langle\alpha|\alpha\rangle}\in\Z$ and
$s_\alpha(\beta):=\beta-k_{\alpha,\beta}\alpha\in \Delta$.
\item[3.]If, for $\alpha\in \Delta$, $\langle\alpha|\alpha\rangle= 0$,
then for every $\beta\in\Delta$ such that $\langle\alpha|\beta\rangle\ne 0$
at least one of $\beta+\alpha$, $\beta-\alpha$, belongs to $\Delta$. 
\end{itemize}
\end{dfn}

\begin{dfn}
\label{dfn:symm}
Let $\Delta\subset V$ be a symmetric subset (that is, satisfying the 
condition $\Delta=-\Delta$). 
We say that $\Delta$ is {\sl indecomposable}
 if $\Delta$ can not be decomposed into a union of two non-empty mutually orthogonal symmetric subsets.
\end{dfn}

Let $\Delta$ be a WGRS. We denote by $\Delta^\iso$ the set of $\alpha\in\Delta$ such that
$\langle\alpha|\alpha\rangle= 0$ {\sl (isotropic roots)} and let $\Delta^\an=\Delta\setminus \Delta^\iso$ {\sl(anisotropic roots)}.
We define the Weyl group $W$ as the group generated by
the reflections $s_\alpha$, $\alpha\in \Delta^\an$.

 \subsection{Properties of WGRS} 
In this subsection (except for \ref{lem:two-step}(3)) a WGRS
$\Delta$ is not required to be finite.

We will use the following easy fact: if $\beta$, $\beta'$ are nonorthogonal isotropic roots and $\gamma:=\beta'-\beta\in\Delta$, then
$\gamma\in\Delta^{\an}$ and $s_{\gamma}\beta=\beta'$.

\begin{lem}\label{lem:non-orthogonal}
Let $\Delta$ be an indecomposable WGRS.
Let $\beta,\beta'$ be mutually orthogonal isotropic roots
and $\beta'\not=\pm\beta$.
Then there exists $\delta\in \Delta^\an$
which is not orthogonal to both 
$\beta$ and $\beta'$.
\end{lem}
\begin{proof}
Since $\Delta$ is irreducible, there is a chain
$$\beta=\gamma_1,\ldots,\gamma_n=\beta'$$
of roots such that 
$\langle\gamma_i|\gamma_{i+1}\rangle\ne 0$. We can assume that $n$
is minimal possible. Then $\langle\gamma_i|\gamma_j\rangle=0$ for
$|i-j|>1$. We will show that one can choose a minimal chain as above such that all $\gamma_i$ are isotropic. Let $\gamma_1,\ldots,\gamma_{i-1}$ be isotropic and $\gamma_i\in \Delta^\an$. Then 
$\gamma_i':=s_{\gamma_i}(\gamma_{i-1})=\gamma_{i-1}-c\gamma_i$ with 
$c\ne 0$, so replacing $\gamma_i$ with $\gamma_i'$ we get another chain
connecting $\beta$ with $\beta'$ with $\gamma_i'$ isotropic.
Note that for a shortest chain of $\gamma_i\in \Delta^\iso$ connecting $\gamma_1$ with $\gamma_n$, the sum $\delta_j=\sum\limits_{i=1}^j\gamma_i$ is a root (we allow ourselves to replace $\gamma_j$ with $-\gamma_j$ if needed). Thus $\delta_n=\sum\limits_{i=1}^n \gamma_i$ and
$\delta_{n-1}=\sum\limits_{i=1}^{n-1} \gamma_i$ are two roots 
which are not orthogonal to $\beta,\beta'$, and 
$\langle\delta_n\mid\delta_n\rangle\not=
\langle\delta_{n-1}\mid\delta_{n-1}\rangle$, so at least 
one of them is not isotropic.
\end{proof}

\begin{lem}
\label{lem:two-step}
Let $\Delta\subset V$ be an indecomposable WGRS and $\beta,\beta'$ be two 
isotropic roots. Let $V_1=\Span(\beta,\beta')$ and $V_0=V_1^\perp$.
\begin{itemize}
\item[1.] If $\langle\beta\mid\beta'\rangle\not=0$, then 
$\beta'=\pm s_{\alpha}\beta$ for some
$\alpha\in\Delta^{\an}$.
\item[2.] If $\langle\beta\mid\beta'\rangle=0$ and $\beta'\not=\pm\beta$, 
then there exist $\alpha_1,\alpha_2\in \Delta^{\an}$ such that
for $w:=s_{\alpha_2}s_{\alpha_1}$ one has  $w(\beta)=\pm\beta'$.
\item[3.] If $\Delta$ is finite, then there exist $\alpha_1,\alpha_2\in \Delta^{\an}$ such that
for $w:=s_{\alpha_2}s_{\alpha_1}$ one has  $w(\beta)=\pm\beta'$,
$w^2=\id$ and $w|_{V_0/V_1}=\id$.
\end{itemize}
\end{lem}

\begin{proof} If $\langle\beta\mid\beta'\rangle\not=0$, then,
changing the signs if necessary we can assume that 
$\alpha=\beta-\beta'$ is a root. Since $\langle\beta\mid\beta'\rangle\not=0$,
$\alpha\in\Delta^{\an}$ and $s_{\alpha}\beta'=\beta$. This gives 1.

Consider the case $\langle\beta\mid\beta'\rangle=0$.
By Lemma~\ref{lem:non-orthogonal} there exists $\delta\in\Delta^{\an}$ which is not orthogonal to both $\beta$ and $\beta'$.
Then $\beta_1:=s_{\delta}\beta$ is an isotropic root 
which is not orthogonal to both $\beta$ and $\beta'$.
Changing the signs if necessary we can assume that 
$\gamma=\beta-\beta_1$ and $\gamma'=\beta'-\beta_1$ are roots.
Since $\beta,\beta_1,\beta'$ are isotropic, $\gamma$ and $\gamma'$
are anisotropic and $s_{\gamma}\beta_1=\beta$, $s_{\gamma'}\beta_1=\beta'$. Therefore $\beta=s_{\gamma'}s_{\gamma}\beta'$. This proves 2.

For 3 note that
$\gamma-\gamma'=\beta-\beta'$  has square length  zero.
Since $\Delta$ is finite, $\Delta^{\an}$ forms a classical root system,
so $\langle\gamma-\gamma'\mid\gamma-\gamma'\rangle=0$ implies 
$\langle\gamma\mid\gamma'\rangle=0$.
Hence $w^2=\id$. Then $w\beta=\beta'$ and $w\beta'=\beta$,
so $wV_1=V_1$ and $wV_0=V_0$. Let $V_2=\Span(\beta,\beta',\beta_1)^\perp$. Then
$V_0=V_2\oplus \C\beta'$. If $v\in V_2$, then $\langle v\mid\gamma\rangle=\langle v\mid\gamma'\rangle=0$, so $wv=v$. Since
$w\beta'-\beta'=\beta-\beta'\in V_1$, one obtains $wv\in v+V_1$ for all $v\in V_0$. This completes the proof.
\end{proof}

\begin{dfn}A subset $S\subset \Delta^\iso$ is called an {\sl iso-set} if
\begin{itemize}
\item[1.] $S=-S$.
\item[2.] $\langle s|t\rangle=0$ for all $s,t\in S$.
\end{itemize}
\end{dfn}

Iso-sets are ordered by inclusion.
For $v\in V$ we denote by $A_v^\mx$ the collection of maximal iso-sets
orthogonal to $v$.

\begin{lem}
\label{lem:trans-isosets}
Let $S,S'\in A^\mx_v$. There exists $w\in\Stab_W(v)$ such that
$S'=w(S)$.
\end{lem}
\begin{proof}We proceed by induction in the cardinality of $S'\setminus S$. Let $\alpha\in S'\setminus S$. Then $S\cup\{\alpha\}$ is not an iso-set, so there exists $\beta\in S$ such that $\langle\alpha|\beta\rangle\ne 0$. This implies that $\alpha\pm\beta\in \Delta$; without loss
of generality we can assume $\gamma=\alpha-\beta\in\Delta$. Obviously,
$\gamma\in \Delta^\an$ so $s_\gamma\in W$. The formula
$
s_\gamma(x)=x-2\langle x|\gamma\rangle/\langle\gamma|\gamma\rangle
$
implies that $s_\gamma$ preserves $v$ and $S\cap S'$, and
$s_\gamma(\alpha)=\beta$.
This implies the claim.
\end{proof}

\begin{crl}
The group $W$ acts transitively on the set of maximal iso-sets.
In particular, all maximal iso-sets have the same cardinality.
\end{crl}\qed

\subsection{Equivalence relation defined by a WGRS}
Given a WGRS $\Delta$, we define an equivalence relation $R$ on $V$ as the one generated by the action of $W$ 
and by the condition
$$
\langle v\mid v+c\alpha\rangle\in R\textrm{ for }\ \alpha\in\Delta^\iso, \langle v|\alpha\rangle=0,\ c\in\C.
$$

Since $R$ is generated by linear Lagrangian relations $\Gamma_{s_\alpha}$, $\alpha\in\Delta^\an$, and
$E_{\alpha^\perp}$, $\alpha\in \Delta^\iso$, $R$ is a Lagrangian
equivalence relation by \ref{prp:monoid}. The following result explicitly describes $R$.

\begin{prp}
Let $v\in V$ and let $S\in A_v^\mx$. Then
$$
(v,v')\in R\Longleftrightarrow v'\in W(v+\Span_\C(S)).
$$
\end{prp}
\begin{proof}
Obviously the elements of $W(v+\Span_\C(S))$ are equivalent to $v$. 
To prove that any $v'$ equivalent to $v$ lies in $W(v+\Span_\C(S))$,
we will verify that for any $v'\in W(v+\Span_\C(S))$ and any 
$\beta\in \Delta^\iso$ such that $\langle v'|\beta\rangle=0$, one has
$v'+c\beta\in W(v+\Span_\C(S))$. We have $v'=w(v+x)$ where $x\in\Span(S)$. Clearly $S\in A_{v+x}^\mx$ so $w(S)\in A^\mx_{v'}$. By 
Lemma~\ref{lem:trans-isosets} there exists $w'\in W$ such that
$w'(v')=v'$ and $w^{\prime-1}(\beta)\in w(S)$ that is $\beta=w'w(\alpha)$
for some $\alpha\in S$. Finally, we have
$$
v'+c\beta=w'(v')+c\beta=w'w(v+x+c\alpha)\in w'w(v+\Span_\C(S)).
$$
\end{proof}

\begin{crl}
\label{crl:L-description}
\begin{itemize}
\item[1.]The components of $R$ have form
$\Gamma_w\circ E_{S^\perp}$ where $S$ is an iso-set and $w\in W$.
\item[1.]The special coisotropic subspaces of $R$ are
$S^\perp$ where $S$ runs through the set of iso-sets in $R$.
\end{itemize}
\end{crl}

\subsection{Semiregularity of $R$}
\label{ss:semiregular}
Let $\Delta$ be a WGRS.
 By Lemma~\ref{lem:two-step}, 
$W$ acts transitively on $\Delta^\iso/\pm 1$. 
Therefore, the equivalence relation $R$  is $1$-semiregular. The following result implies that it is in fact  semiregular.

\begin{prp}
\label{prp:RtoR'}Let $\Delta$ be a WGRS in $V$ and let $\alpha\in \Delta^\iso$. Denote
$V'=V_0/V_1$ where $V_1=\C\alpha$ and $V_0=V_1^\perp$ and let $\Delta'$
be the image of $\Delta\cap V_0$ in $V'$ (we exclude $\pm\alpha$ whose image in $V'$ is zero). Then $\Delta'$ is a WGRS in $V'$ and the corresponding
equivalence relation $R'$ on $V'$ is the one induced on $V'$ by $R$.
\end{prp}
\begin{proof}

Obviously $\Delta'=-\Delta'$ and the projection $V_0\to V_0/V_1$ preserves the 
bilinear form (it is degenerate on $V_0$). The conditions 1 --- 3
of Definition~\ref{dfn:WGRS} obviously hold, so $\Delta'$ is a WGRS.
Let us denote (temporarily) by $R''$ the equivalence relation on $V'$
induced by $R$. It remains to verify that the relation $R'$
defined by $\Delta'$ coincides with $R''$.

By~\ref{crl:L-description} all components of $R$ have form
$E_{S^\perp}\circ\Gamma_w$ where $S$ is an iso-set and $w\in W$.
By Proposition~\ref{prp:red}, such a component gives rise to an irreducible component of $R''$ iff
$S^\perp\subset\alpha^\perp$ and $w(S^\perp)\subset\alpha^\perp$, that 
is, iff $\{\alpha,w^{-1}(\alpha)\}\in S$.

If $\alpha=w^{-1}(\alpha)$, this
means that $w\in\Stab(\alpha^\perp)$ and $\alpha\in S$, so that the 
component
$E_{S^\perp}\circ\Gamma_w$ determines a component of $R'$. 
If $\alpha\ne w^{-1}(\alpha)$, 
there exists, by Lemma~\ref{lem:two-step}, an element $y\in W$ such that
$y(\alpha)=w^{-1}(\alpha)$ and $y(w^{-1}(\alpha))=\alpha$, and such that
$y$ acts trivially on the quotient $(\alpha,w^{-1}(\alpha))^\perp/
\Span(\alpha,w^{-1}(\alpha))$. This implies that $y$ acts also trivially
on the subquotient $S^\perp/\Span(S)$, so that 
$E_{S^\perp}\circ\Gamma_y=E_{S^\perp}$. Finally,
$E_{S^\perp}\circ\Gamma_w=E_{S^\perp}\circ\Gamma_y\circ\Gamma_{wy^{-1}}=
E_{S^\perp}\circ\Gamma_{wy^{-1}}$ and $wy^{-1}(\alpha)=\alpha$ which reduces this case to the case $w(\alpha)=\alpha$.
This proves the claim.
\end{proof}

\end{document}